\documentclass[11pt]{article}

\usepackage{amsmath}
\usepackage{amssymb}
\usepackage{amsthm}
\usepackage{url}
\usepackage{enumerate}
\usepackage[pdftex, hidelinks]{hyperref}

\title{A probabilistic model of X-ray computed tomography}
\author{
  Tyler Gomez\\
  University of Central Florida\\
  Orlando, FL\\
  \url{tgomez@ea.com}
    \and
  Jason Swanson\\
  University of Central Florida\\
  Orlando, FL\\
  \url{jason.swanson@ucf.edu}
    \and
  Alexandru Tamasan\\
  University of Central Florida\\
  Orlando, FL\\
  \url{tamasan@math.ucf.edu}
}

\newtheorem{thm}{Theorem}[section]
\newtheorem{cor}[thm]{Corollary}
\newtheorem{prop}[thm]{Proposition}
\newtheorem{lemma}[thm]{Lemma}
\theoremstyle{definition}
\newtheorem{rmk}[thm]{Remark}

\numberwithin{equation}{section}

\providecommand{\ang}[1]{\langle{#1}\rangle}
\providecommand{\flr}[1]{\lfloor{#1}\rfloor}

\DeclareMathOperator{\Lip}{Lip}
\DeclareMathOperator{\Pois}{Poisson}

\def\To{\Rightarrow}
\def\ol{\overline}
\def\wh{\widehat}
\def\wt{\widetilde}

\def\be{\beta}
\def\ga{\gamma}
\def\De{\Delta}
\def\ep{\varepsilon}
\def\ze{\zeta}
\def\th{\theta}
\def\ka{\kappa}
\def\la{\lambda}
\def\si{\sigma}
\def\ph{\varphi}
\def\Om{\Omega}

\def\bC{\mathbb{C}}
\def\bN{\mathbb{N}}
\def\bR{\mathbb{R}}
\def\bS{\mathbb{S}}
\def\cB{\mathcal{B}}
\def\cD{\mathcal{D}}

\begin{document}

\maketitle

\begin{abstract}
  We consider a discrete stochastic process, indexed by lines through the unit
  disk in the plane, which models the observed photon counts in a medical X-ray
  tomography scan. We first prove a functional law of large numbers, showing
  that this process converges in $L^2$ to the X-ray transform of the underlying
  attenuation function. We then prove a family of functional central limit
  theorems, which show that the normalized observations converge to a white
  noise on the space of lines, provided the growth rate of the mean number  of
  photons per line is greater than a certain power of the number of lines
  scanned. Using this family of theorems, we can reduce that power arbitrarily
  close to zero by adding correction terms to the normalization. We also prove a
  Berry-Esseen inequality that gives a concrete rate of convergence for each
  functional central limit theorem in our family of theorems.

  \bigskip


  \smallskip


\end{abstract}

\section{Introduction}

Medical X-ray tomography, or more commonly, computed tomography (CT), uses
penetrating beams of X-ray photons to construct internal images of the body.
When an X-ray beam passes through the body along a line $L$, its intensity
decreases according to the Beer-Lambert (for brevity, B-L) law,
\begin{equation*}
  I = I_0 \exp \left( - \int_L f(x) \, dx \right),
\end{equation*}
where $I_0$ is the initial intensity, $I$ is the final intensity, and $f$
encodes the attenuation property of the inner body. In $X$-ray tomography one
determines $f$ from knowledge of $I/I_0$  along many different lines. Originated
in the seminal work \cite{Radon1917}, classical mathematics of tomography
concerns questions (e.g., uniqueness, stability, accuracy, resolution,
reconstructing algorithms) in the problem of recovering $f$  from its line
integrals. These have been well understood by now; an insight in the breadth and
depths of the subject can be gained from, e.g.,
\cite{Natterer2001,Epstein2008,Kuchment2014}.

In practice, however, many problems occur due to the inadequacy of the B-L law.
One such problem is the so called ``beam hardening" effect thought due to the
insufficient modeling of the dependence of $f$ on the energy of the $X$-ray,
where the measurement is, in fact,
\begin{equation*}
  I = I_0\int \sigma(E) \exp \left( - \int_L f(x,E) \, dx \right)d E,
\end{equation*}
for a range of distribution in the energy of the $X$-ray beam. Mitigation
approaches exist and essentially rely on ad-hoc modeling: For example in
\cite{Stonestrom1981} a specific ansatz for $f(x,E)$ is sufficient to capture
the qualitative feature that the lower the energy, the worse the line integral
approximation gets, but not in a quantitative way. Different mitigation
approaches  (e.g., \cite{Whiting2002, Elbakri2003, LaRiviere2006,
Xu2009,Yun2025}) modify the particle transport along lines by adding a
stochastic perturbation and study the effect of this noise to the corresponding
solution. All of these mitigation approaches are at the modeling level.

In this paper, we propose a different approach based solely on the number of
particles, irrespective of the physical causes producing them. We think of the
B-L law as fundamentally probabilistic: While the beam intensity decreases
because X-ray photons are absorbed by the body, the actual number of absorbed
photons is random. Thus, the Beer-Lambert equation tells us about the mean
number of absorbed particles. In order to be reliable, the initial number of
photons must be large.

For a lower number of photons, we interpret the fluctuations in the B-L law as a
type of central limit theorem. The effect of a smaller number of incident
photons on the variance from the expected number of emerging photons was
observed as far back as the 1970s: in \cite{Gore1978}, Gore and Tofts gave a
careful derivation of this variance in a certain model of image reconstruction.

We typically focus on two-dimensional cross-sections of an object. For
simplicity, then, we take our attenuation function, $f(x)$, to be defined on the
unit disk $D$ in the plane. If $y$ is a line through $D$, then $Xf(y)$ denotes
the line integral of $f$ along $y$. The function $Xf$ is called the X-ray
transform of $f$. It can be shown that if an X-ray photon passes through the
object along the line $y$, then the probability it emerges without being
absorbed is $e^{-Xf(y)}$. Imagine we pass a large number, $N$, of X-ray photons
through the body along the line $y$. We can then count the number of photons
that emerge and denote it by $S^y f$. By the law of large numbers, $S^y f/N
\approx e^{-Xf(y)}$. Hence, if we define $Y_N f(y) = -\log(S^y f/N)$, then $Y_N
f$ is an approximation of $Xf$. If $N$ is very large, we can treat this
approximate equality as an equality and think of $Y_N f$ as if it were a direct
measurement of $Xf$. With $Xf$ in hand, we must then invert the X-ray transform
to determine $f$ and obtain our picture of the inside of the body.

Now, in practice, we can only do this for a finite collection of lines, $\{y_{j,
k}: 1 \le j \le n, 1 \le k \le m\}$. This means we would be using a random
simple function, $Y_{n, m, N}f$ to approximate a deterministic simple function,
$X_{n, m}f$, which is itself an approximation of $Xf$. We see, therefore, that
the mathematics of medical X-ray tomography involves three separate problems:
\begin{enumerate}[(P1)]
  \item Use $Y_{n, m, N}f$ to approximate $X_{n, m}f$.
  \item Use $X_{n, m}f$ to approximate $Xf$.
  \item Invert $Xf$ to obtain $f$.
\end{enumerate}
Problems (P2) and (P3) are problems of classical, deterministic analysis.
Problem (P1), though, is probabilistic.

The traditional approach to Problem (P1) is to take $N$ so large that we may
treat $Y_{n, m, N}f$ and $X_{n, m}f$ as if they are equal and move on to
Problems (P2) and (P3). This approach amounts to a kind of functional law of
large numbers. In Theorem \ref{T:LLN}, we make this precise by showing that
\begin{equation*}
  E\|Y_{n, m, N}f - X_{n, m}\|_2 \le CN^{-1/2}.
\end{equation*}
Here, $\|\cdot\|_2$ denotes the $L^2$ norm in $Z$, the space of lines through
the unit disk. In Proposition \ref{P:stepEst}, however, we see that the error in
Problem (P2) is of magnitude $n^{-1}$, at least in the special case that $n
\approx m$. Hence, when using this traditional approach to (P1), we must take $N
\approx n^2$ for the errors in (P1) and (P2) to be comparable.

To improve on this, we must understand the fluctuations of $Y_{n, m, N}f$ around
its asymptotic limit. In other words, we aim to prove a functional central limit
theorem for $Y_{n, m, N}f$. Our main result, Theorem \ref{T:CLT}, does exactly
this. In fact, it gives an entire family of functional central limits. The
simplest member of this family says that
\begin{equation*}
  \sqrt{nmN} \, (Y_{n, m, N}f - X_{n, m}f) \To \pi e^{Xf/2} \, W
\end{equation*}
in $\cD'(Z)$ whenever $n, m, N \to \infty$ in such a way that $nm/N \to 0$.
Here, the double arrow $\To$ denotes convergence in distribution, the object $W$
is a white noise on $Z$, and $\cD'(Z)$ is the space of generalized functions on
$Z$. The problem with this most basic result is that it does not allow us to
reduce the value of $N$. In fact, it requires us to increase it. Assuming again
that $n \approx m$, we see that we must now take $N \gg n^2$. Compared to the
functional law of large numbers that requires $N \approx n^2$, this is worse.

However, the second simplest instance of Theorem \ref{T:CLT} does give us an
improvement. This is presented in Corollary \ref{C:CLT3}. There we find that
\begin{equation}\label{CLT3}
  \sqrt{nmN} \left(
    Y_{n, m, N}f - X_{n, m}f + \frac 1 {2N} e^{X_{n, m}f}
  \right) \To \pi e^{Xf/2} \, W
\end{equation}
in $\cD'(Z)$ whenever $n, m, N \to \infty$ in such a way that $nm/N^3 \to 0$. By
adding a correction term to our centering of $Y_{n, m, N}f$, we improve the rate
of convergence. Now we need only take $N \gg n^{2/3}$, which is a significant
improvement on needing $N \approx n^2$ in the functional law of large numbers.

More generally, Theorem \ref{T:CLT} shows us that by adding additional
correction terms, we can obtain arbitrarily sharper results. As an example, in
Corollary \ref{C:CLT5}, we see that
\begin{equation}\label{CLT5}
  \sqrt{nmN} \left(
    Y_{n, m, N}f - X_{n, m}f + \frac 1 {2N} e^{X_{n, m}f}
    + \frac 1 {12N^2} e^{2 X_{n, m}f}
  \right) \To \pi e^{Xf/2} \, W
\end{equation}
in $\cD'(Z)$ whenever $n, m, N \to \infty$ in such a way that $nm/N^5 \to 0$. By
adding more and more correction terms, we can obtain our result under the
condition $nm/N^\ka \to 0$ for any odd, positive integer $\ka$.

Theorem \ref{T:CLT} is proved in Section \ref{S:CLTpf} by first establishing the
Berry-Esseen inequality that appears in \eqref{BE}. This inequality, which is
tailored specifically for the tomography model we are considering, is itself a
consequence of the structurally simpler and more familiar-looking Berry-Esseen
inequality in \eqref{BEraw}. The latter inequality follows more or less directly
from the results in \cite{Petrov1975} for sums of independent, mean zero random
variables. (The value of the constant was obtained in \cite{Shevtsova2014}.)

The path of derivation that leads from \eqref{BEraw} to Theorem \ref{T:CLT}
involves the introduction of two different Taylor expansions. Our use of
expansions, however, should not be confused with the kind of expansions one
finds, for example, in \cite{Bhattacharya2010} and \cite{Barbour1986}. In those
cases as well as ours, one is concerned with a sum $W_n$ of independent random
variables, and how close its distribution is to that of a normal random variable
$Z$. The basic Berry-Esseen inequalities in \cite{Petrov1975} give us
first-order expansions of the form $P(W_n \le x) = \Phi(x) + O(n^{-1/2})$. In
\cite{Bhattacharya2010}, it is this expansion that is refined, so that we can
better estimate the error $P(W_n \le x) - \Phi(x)$ with higher-order precision.
In \cite{Barbour1986}, Stein's method is used to similarly refine the
first-order expansion $E h(W_n) = E h(Z) + O(n^{-1/2})$, where $h$ is a smooth
function. In short, it is the distribution itself that is being expanded in
these approaches.

In our case, we are not expanding the distribution of the sum, but rather the
terms within the sum. In our tomography model, we are working with a sum that
has roughly the form $\sum \ph(S_n)$, where each term is wrapped in a common
function $\ph$. It is this function that we expand. Through expansions, we are
able to identify a succession of better and better centering terms in order to
increase the size of the regime in which we have convergence. As such, our
expansions are less analogous to those in \cite{Bhattacharya2010} and
\cite{Barbour1986}, and more analogous to the way we use expansions in the
construction of the stochastic integral.

In all of this, there is one subtle issue that we have yet to address. When we
count the emerging photons along a given line, there is a negligible---but still
positive---probability that the count will be zero. In that case, we cannot
evaluate the logarithm of the photon count. Although this unlikely possibility
may not occur in any given instance of a CT scan, it must be accounted for
mathematically. We present here three different ways to normalize the photon
counts before applying the logarithm.
\begin{enumerate}[(N1)]
  \item Add one to the photon count, regardless of its value.
  \item Add one to the photon count when it is zero.
  \item If the photon count along a line is zero, discard the result for that
        line and try again.
\end{enumerate}
The first of these is the simplest, mathematically, while the last is perhaps
the most natural in practice. We give complete results for all three
normalization methods in this paper. The limits in \eqref{CLT3} and \eqref{CLT5}
hold when using method (N1), whereas under (N2) and (N3), the coefficients
$1/2N$ and $1/12N^2$ become $-1/2N$ and $7/12N^2$, respectively. (See Corollary
\ref{C:maxCLT} for details.)

The outline of this paper is as follows. Section \ref{S:main} establishes
important definitions and notation, based on (N1), and presents our main
results. In Section \ref{S:LLNpf}, we prove the functional law of large numbers.
To prove our Berry-Esseen inequality, we must derive the asymptotics of certain
parameters. This is done in Section \ref{S:BEparams}. Then, in Section
\ref{S:CLTpf}, we prove our functional central limit theorem. Finally, in
Section \ref{S:altNorm}, we show how these results are altered under the
normalization methods, (N2) and (N3).

\section{Notation, background, and main results}\label{S:main}

\subsection{Discretizing the X-ray transform}

Let $D = \{x \in \bR^2: |x| < 1\}$. Let $C(\ol D; [0, \infty))$ be the set of
continuous, nonnegative functions on $\ol D$ and $\Lip(D)$ the set of Lipschitz
functions on $D$. Let $Z = (0, 1) \times \bS^1$, where $\bS^1$ denotes the unit
circle in $\bR^2$. For $y = (s, \tau) \in Z$, let $L_y = \{x \in \bR^2: \ang{x,
\tau} = s\}$ denote the line through the point $s \tau \in D$ that is orthogonal
to $\tau$. Given $g \in L^2(D, \cB_D, \la)$, where $\cB_D$ is the Borel
$\si$-algebra and $\la$ is Lebesgue measure, the X-ray transform of $g$ is the
function $Xg: Z \to \bR$ given by
\begin{equation*}
  Xg(y) = \int_{L_y\cap D} g \, d\la_{L_y},
\end{equation*}
where $\la_{L_y}$ is Lebesgue measure on $L_y$. If we define the measure $\nu$
on $Z$ by
\begin{equation*}
  \nu(ds \, d\tau) = (1 - s^2)^{-1/2} \, ds \, \si(d\tau),
\end{equation*}
where $\si$ is Lebesgue measure on $\bS^1$, then $Xg \in L^2(Z, \cB_Z, \nu)$,
and $X: L^2(D) \to L^2(Z)$ is a bounded linear operator. Note that $\nu(Z) =
\pi^2$. Also note that $|Xg(y)| \le \|g\|_\infty \la_{L_y}(L_y \cap D)$. Hence,
$\|Xg\|_\infty \le 2\|g\|_\infty$. In particular, $X: L^\infty(D) \to
L^\infty(Z)$ is continuous, and also $X: C(\ol D; [0, \infty)) \to C(\ol Z; [0,
\infty))$ is continuous. We further note that if $g \in \Lip(D)$, then $Xg \in
\Lip(Z)$. Throughout this work, we will use $\|\cdot\|_r$ to denote the norm in
$L^r(Z, \cB_Z, \nu)$. We will also use unadorned brackets, $\ang{\cdot, \cdot}$,
for the inner product in $L^2(Z)$.

Fix a strictly positive function $f \in C(\ol D; [0, \infty)) \cap \Lip(D)$ to
serve as our attenuation function. Given $n, m \in \bN$, let $s_j = \sin(\pi
j/2n)$, $\th_k = 2\pi k/m$, and
\begin{equation*}
  A_{j, k} = (s_{j - 1}, s_j] \times \{e^{i\th}: \th \in (\th_{k - 1}, \th_k]\}
    \subset Z.
\end{equation*}
Here we are identifying $\bC$ with $\bR^2$. Note that $\nu(A_{j, k}) = \pi^2/nm$
and $\De s_j = s_j - s_{j - 1} \le \pi/2n$. Let $y_{j, k} = (s_j, e^{i\th_k})$
denote one of the corners of $A_{j, k}$, and define
\begin{equation*}
  X_{n, m}f = \sum_{j, k} Xf(y_{j, k}) 1_{A_{j, k}}.
\end{equation*}
In the above, we use $\sum_{j, k}$ to denote the double sum $\sum_{j = 1}^n
\sum_{k = 1}^m$. We will employ this convention throughout this paper. We will
also use $p_{j, k} = e^{-Xf(y_{j, k})}$ to denote the probability of a given
photon emerging from the body after being transmitted along the line $L_{y_{j,
k}}$. Note that $0 < e^{-\|Xf\|_\infty} \le p_{j, k} \le 1$.

Another convention used throughout is that the letter $C$ will denote a constant
which may change value from instance to instance. The constant $C$ will
typically depend only on a limited number of parameters. Which parameters these
are will be made clear whenever such a constant is used.

\begin{prop}\label{P:stepEst}
  There exists a constant $C$ that depends only on $f$ such that
  \begin{equation*}
    \|X_{n, m}f - Xf\|_\infty \le C \left( \frac 1 n + \frac 1 m \right)
  \end{equation*}
  for all $n$ and $m$.
\end{prop}

\begin{proof}
  Fix $y \in Z$ and choose $j$ and $k$ such that $y \in A_{j, k}$. Then $X_{n,
  m}f(y) - Xf(y) = Xf(y_{j, k}) - Xf(y)$. The result now follows from the fact
  that $Xf \in \Lip(Z)$ and the diameter of $A_{j, k}$ is bounded by $C(\De s +
  2\pi/m) \le C(\pi/2n + 2\pi/m)$.
\end{proof}

\subsection{Defining the observation process}

In practice, we cannot control the number of X-ray photons that are transmitted
along a given line. We can only control its mean. The actual number of particles
will be random and have a Poisson distribution. Hence, we fix $N \in \bN$, and
for each $y \in Z$, we let $V^y \sim \Pois(N)$. Then $V^y$ represents the number
of X-ray photons that are sent along the line $L_y$. For each $\ell \in \bN$,
let $\xi^y_\ell f$ be Bernoulli with $P(\xi^y_\ell f = 1) = e^{-Xf(y)}$. The
event $\{\xi^y_\ell f = 1\}$ represents the event that the $\ell$-th particle to
pass along $L_y$ emerges from the body. Because the attenuation function $f$ is
fixed, the absorption of distinct particles is independent. Therefore, $V^y,
\xi^y_1 f, \xi^y_2 f, \ldots$ are independent from each other and also
independent for different values of $y$. Finally, define $S^y f = \sum_{\ell =
1}^{V^y} \xi^y_\ell f$, which represents the number of particles that emerge
after being sent along $L_y$. We then have $S^y f \sim \Pois(N e^{-Xf(y)})$.
(This latter fact follows, for instance, from \cite[Exercise
3.6.12]{Durrett2010}.)

As noted in the introduction, we would like to define $Y_{n, m, N}f$, which is
our approximation of $X_{n, m}f$, to be
\begin{equation*}
  - \sum_{j, k} {
    \log \left( \frac{S^{y_{j, k}}f} N \right) 1_{A_{j, k}}
  },
\end{equation*}
since $S^{y_{j, k}}f/N \approx e^{-Xf(y_{j, k})}$, according to the law of large
numbers. But this is not possible. Although the event $\{S^{y_{j, k}}f = 0\}$ is
exponentially unlikely, it still has positive probability. Hence, the random
variable $\log(S^{y_{j, k}}f)$ is undefined, at least as a real-valued random
variable. We could define $\log(S^{y_{j, k}}f) = -\infty$ on $\{S^{y_{j, k}}f =
0\}$, but then $Y_{n, m, N}$ would be infinity on a set of positive measure,
with positive probability. This would cause a number of problems that prevent us
from making meaningful statements about the asymptotics of $Y_{n, m, N}$. For
example, it would imply that $E\|Y_{n, m, N}\|_2 = \infty$.

To fix this, we must make some adjustment to account for the remote possibility
that $S^{y_{j, k}}f = 0$. There are a number of different adjustments we might
make. The most natural would be to simply try again. That is, if we happened to
find that no particles emerged along a certain line, then we would redo the scan
along that line until at least one particle emerges. To represent this, we could
let $\wt S^{y_{j, k}}f$ have the same distribution as $S^{y_{j, k}}f$
conditioned to be positive, and then define $Y_{n, m, N}f$ as
\begin{equation*}
  - \sum_{j, k} {
    \log \left( \frac{\wt S^{y_{j, k}}f} N \right) 1_{A_{j, k}}
  }.
\end{equation*}
This is normalization method (N3) listed in the introduction.

Another possibility would be to not redo the scan but to simply record a count
of zero particles as if it were a count of one particle. This would be
represented by defining $Y_{n, m, N}f$ as
\begin{equation*}
  - \sum_{j, k} {
    \log \left( \frac{S^{y_{j, k}}f \vee 1} N \right) 1_{A_{j, k}}
  },
\end{equation*}
where $a \vee b$ is the maximum of $a$ and $b$. This is normalization method
(N2) listed in the introduction.

The approach we take initially in this paper is (N1), which is to simply add one
to our emerging photon count, no matter what happens. That is, we define
\begin{equation}\label{Ydef}
  Y_{n, m, N}f = - \sum_{j, k} {
    \log \left( \frac{S^{y_{j, k}}f + 1} N \right) 1_{A_{j, k}}
  },
\end{equation}
We have chosen to begin with this approach because of its mathematical
simplicity. In Section \ref{S:altNorm}, we give a complete analysis of the other
two approaches, (N2) and (N3).

\subsection{Poisson estimates}

To deal with the Poisson distribution in the definition of $Y_{n, m, N}f$, we
define $\mu_r(\la) = E[(S - \la)^r]$, where $S \sim \Pois(\la)$ and $r$ is a
nonnegative integer. The functions $\mu_r$ satisfy the recursion $\mu_{r + 1} =
\la(\mu_r' + r \mu_{r - 1})$. (See, for example, \cite[(4.8)]{Riordan1937}.) The
first few such functions are $\mu_0(\la) = 1$, $\mu_1(\la) = 0$,
\begin{align*}
  \mu_2(\la) &= \la,\\
  \mu_3(\la) &= \la,\\
  \mu_4(\la) &= 3\la^2 + \la, \text{ and}\\
  \mu_5(\la) &= 10\la^2 + \la.
\end{align*}
More generally, it can be shown that $\mu_r$ is a polynomial of degree
$\flr{r/2}$ with no constant term.

\begin{lemma}\label{L:PoisMom}
  Let $r \in [0, \infty]$. Then there exists a constant $C$ that depends only on
  $f$ and $r$ such that
  \begin{equation*}
    E|S^{y_{j, k}}f - Np_{j, k}|^r \le C N^{r/2} p_{j, k}^{r/2}.
  \end{equation*}
\end{lemma}

\begin{proof}
  Let $i$ be the smallest integer such that $r \le 2i$. Since $\mu_{2i}(\la) =
  O(\la^i)$ as $\la \to \infty$ and $p_{j, k} \ge e^{-\|Xf\|_\infty}$, we have
  \begin{equation*}
    E|S^{y_{j, k}}f - Np_{j, k}|^{2i} = \mu_{2i}(Np_{j, k}) \le C N^i p_{j, k}^i.
  \end{equation*}
  By Jensen's inequality,
  \begin{equation*}
    E|S^{y_{j, k}}f - Np_{j, k}|^r \le (E|S^{y_{j, k}}f - Np_{j, k}|^{2i})^{r/2i}
     \le C N^{r/2} p_{j, k}^{r/2},
  \end{equation*}
  where $C$ depends only on $f$ and $r$.
\end{proof}

\subsection{The functional limit theorems}

To express our functional central limit theorem, we use the following notation.
Let $\cD'(Z)$ be the space of generalized functions on $Z$. That is, $\cD'(Z)$
is the space of continuous linear functionals on $C_c^\infty(Z)$, endowed with
the topology of pointwise convergence. A white noise on $Z$ is a
$\cD'(Z)$-valued random variable $W$ that satisfies $\ang{W, g} \sim N(0,
\|g\|^2)$ for all $g \in C_c^\infty(Z)$. The existence of white noise is
guaranteed by the Bochner-Minlos theorem. (See, for example, \cite[Theorem
2.1.1]{Holden1996}.) We will use the double arrow, $\To$, to indicate
convergence in distribution.

The main results of this paper are the following two theorems. The first is a
functional law of large numbers and the second is a functional central limit
theorem.

\begin{thm}\label{T:LLN}
  There is a constant $C$ that depends only on $f$ such that
  \begin{equation*}
    \|Y_{n, m, N}f - X_{n, m}f\|_{L^2(S \times \Om)} \le \frac C {\sqrt{N}}
  \end{equation*}
  for all $n$, $m$, and $N$.
\end{thm}

\begin{thm}\label{T:CLT}
  Let $a$ and $b$ be nonnegative integers with $a \ge 1$, and let $W$ be a white
  noise on $Z$. Define
  \begin{equation*}
    Z_{n, m, N}f = \sqrt{nmN} \, \bigg(
      Y_{n, m, N}f - X_{n, m}f - \sum_{r = 1}^b {
        \frac {(-1)^r} {rN^r} \, e^{rX_{n, m}f}
      } - \sum_{r = 2}^a \frac {
        (-1)^r \mu_r(N e^{-X_{n, m}f})
      } {
        r (N e^{-X_{n, m}f} + 1)^r
      }
    \bigg),
  \end{equation*}
  and let $\ka = \min\{a, 2b + 1\}$. Then $Z_{n, m, N}f \To \pi e^{Xf/2} \, W$
  in $\cD'(Z)$ whenever $n, m, N \to \infty$ in such a way that $nm/N^\ka \to
  0$.
\end{thm}

The simplest instance of Theorem \ref{T:CLT} is obtained by taking $a = 1$ and
$b = 0$. In this case, we find that $\sqrt{nmN} \, (Y_{n, m, N}f - X_{n, m}f)
\To \pi e^{Xf/2} \, W$ whenever $n, m, N \to \infty$ in such a way that $nm/N
\to 0$. As noted in the introduction, this result is not strong enough to reduce
the value of $N$ below what is needed for the law of large numbers. In Section
\ref{S:CLTpf}, we illustrate deriving the next two simplest instances by proving
the following corollaries.

\begin{cor}\label{C:CLT3}
  Let $W$ be a white noise on $Z$. Then \eqref{CLT3} holds in $\cD'(Z)$ whenever
  $n, m, N \to \infty$ in such a way that $nm/N^3 \to 0$.
\end{cor}

\begin{cor}\label{C:CLT5}
  Let $W$ be a white noise on $Z$. Then \eqref{CLT5} holds in $\cD'(Z)$ whenever
  $n, m, N \to \infty$ in such a way that $nm/N^5 \to 0$.
\end{cor}

Corollary \ref{C:CLT3} is obtained from Theorem \ref{T:CLT} by taking $a = 3$
and $b = 1$, which gives $\ka = 3$. A priori, it seems like we could obtain an
even simpler result by taking $a = 2$ and $b = 1$, which would give $\ka = 2$.
But as explained at the end of Section \ref{S:CLTpf}, this would only prove the
same result as \eqref{CLT3}, but under the weaker condition that $nm/N^2 \to 0$.
Hence, there is no benefit to taking $\ka = 2$. In fact, there is no benefit to
taking $\ka$ to be any even integer, so that $\ka = 5$ in Corollary \ref{C:CLT5}
is indeed the next non-redundant instance.

\subsection{A Berry-Esseen inequality}

The proofs of Theorems \ref{T:LLN} and \ref{T:CLT} use a combination of two
Taylor expansions to decompose $Y_{n, m, N}f$ into a sum of several terms. For
the central limit theorem, the most important term is
\begin{equation}\label{discWN}
  W_{n, m, N}f = \sqrt{nmN} \, \sum_{j, k} \frac{
    Np_{j, k} - S^{y_{j, k}}f
  }{
    Np_{j, k} + 1
  } \, 1_{A_{j, k}}.
\end{equation}
It is this term which converges to $\pi e^{Xf/2} \, W$ as $n, m, N \to \infty$.
The difference between it and $Z_{n, m, N}f$ tends to $0$, and it is this
difference which requires the extra condition $nm/N^\ka \to 0$.

We prove that $W_{n, m, N}f \To \pi e^{Xf/2} \, W$ by using Berry-Esseen
inequalities of the type found, for example, in \cite{Petrov1975}. To apply
these bounds, we must define and study two key parameters associated with $W_{n,
m, N}f$. For the first, let $g \in C_c^\infty(Z)$. Then $\ang{W_{n, m, N}f, g}$
is a sum of independent, mean zero random variables. Hence, its variance is
\begin{equation}\label{Wvar}
  \si_{n, m, N}^2 = \sum_{j, k} {
    \frac{nmN^2 p_{j, k}}{(Np_{j, k} + 1)^2} \, \ga(A_{j, k})^2
  },
\end{equation}
where $\ga$ is the signed measure defined by $d\ga = g \, d\nu$. The second
parameter we need is
\begin{equation}\label{Wskew}
  L_{n, m, N} = \si_{n, m, N}^{-3} \sum_{j, k} {
    \frac {
      (nmN)^{3/2} E|S^{y_{j, k}}f - Np_{j, k}|^3
    } {
      (Np_{j, k} + 1)^3
    } |\ga(A_{j, k})|^3
  }.
\end{equation}
By \cite[Theorem V.3]{Petrov1975} and \cite{Shevtsova2014}, we have the
Berry-Esseen inequality,
\begin{equation}\label{BEraw}
  |P(\si_{n, m, N}^{-1} \ang{W_{n, m, N}f, g} \le x) - \Phi(x)|
    \le 0.5583 L_{n, m, N},
\end{equation}
for all $x \in \bR$, where $\Phi$ is the standard normal distribution function.
With this as our starting point, we derive in Section \ref{S:CLTpf} the
inequality,
\begin{equation}\label{BE}
  |P(\ang{Z_{n, m, N}f, g} \le x) - P(\|\pi e^{Xf/2} g\|_2 \eta \le x)|
    \le C \left(
      \frac 1 {\sqrt{nm}} + \frac 1 n + \frac 1 m + \frac 1 {N^{1/3}}
      + \left( \frac {nm} {N^\ka} \right)^{1/3}
    \right),
\end{equation}
where $\eta$ is a standard normal and $C$ is a constant that depends only on
$a$, $b$, $f$, and $g$. By tracing through the various proofs, the constant $C$
in \eqref{BE} can be computed in terms of the $L^\infty$ norms and Lipschitz
constants of $f$ and $g$. In this way, one obtains a Berry-Esseen-type
inequality for $Z_{n, m, N}f$.

\section{Proving the functional law of large numbers}\label{S:LLNpf}

The proofs of Theorems \ref{T:LLN} and \ref{T:CLT} use a combination of two
different Taylor expansions to produce a single expansion of $Y_{n, m, N}f$ that
has two integral remainder terms. One remainder is deterministic and the other
is stochastic. Lemma \ref{L:remEst} below controls the stochastic remainder and
is used both here and in Section \ref{S:CLTpf}.

To prove Lemma \ref{L:remEst}, we need of the following large deviations
inequality.

\begin{lemma}\label{L:PoisLD}
  If $S \sim \Pois(\la)$, then
  \begin{equation*}
    P(S \le \la/2) \le \exp\left(-\frac{1 - \log 2}2 \, \la\right).
  \end{equation*}
\end{lemma}

\begin{proof}
  For $t > 0$ and $a \in (0, 1)$, we have
  \begin{equation*}
    e^{\la(e^{-t} - 1)} = E[e^{-tS}] \ge E[e^{-tS} 1_{\{S \le a \la\}}]
    \ge e^{-a \la t} P(S \le a \la).
  \end{equation*}
  Hence, $P(S \le a \la) \le e^{\la(a t + e^{-t} - 1)}$. The right-hand side is
  minimized when $t = -\log a$, or $a = e^{-t}$. Thus,
  \begin{equation*}
    P(S \le a \la) \le e^{-\la(a \log a - a + 1)}.
  \end{equation*}
  Taking $a = 1/2$ finishes the proof.
\end{proof}

\begin{lemma}\label{L:remEst}
  Let $a \in \bN$ and define
  \begin{equation*}
    R = \sum_{j, k} \int_{Np_{j, k}}^{S^{y_{j, k}}f} {
      \frac {(S^{y_{j, k}}f - t)^a} {(t + 1)^{a + 1}}
    } \, dt \, 1_{A_{j, k}}.
  \end{equation*}
  Then there exists a constant $C$ that depends only on $a$ and $f$ such that
  $\|R\|_{L^2(S \times \Om)} \le C N^{-(a + 1)/2}$.
\end{lemma}

\begin{proof}
  Let us define
  \begin{equation*}
    U_{j, k} = \int_{Np_{j, k}}^{S^{y_{j, k}}f} {
      \frac {(S^{y_{j, k}}f - t)^a} {(t + 1)^{a + 1}} \, dt
    }.
  \end{equation*}
  For notational simplicity, let $S = S^{y_{j, k}}f$, $p = p_{j, k}$, and $U =
  U_{j, k}$. If $S > Np/2$, then
  \begin{equation*}
    |U| \le \left(\frac 2 {Np}\right)^{a + 1} \int_{Np}^S |S - t|^a \, dt
      = \frac {C |S - Np|^{a + 1}} {N^{a + 1} p^{a + 1}}
  \end{equation*}
  Hence, by Lemma \ref{L:PoisMom},
  \begin{equation*}
    E[|U|^2 1_{\{S > Np/2\}}]
      \le \frac C {N^{2a + 2} p^{2a + 2}} \, E|S - Np|^{2a + 2}
      \le \frac C {N^{a + 1} p^{a + 1}}.
  \end{equation*}
  On the other hand,
  \begin{equation*}
    |U| \le \int_{Np}^S |S - t|^a \, dt = C |S - Np|^{a + 1}.
  \end{equation*}
  Thus,
  \begin{equation*}
    E|U|^4 \le C \, E|S - Np|^{4a + 4} \le C N^{2a + 2} p^{2a + 2},
  \end{equation*}
  which, by Lemma \ref{L:PoisLD}, gives
  \begin{align*}
    E[|U|^2 1_{\{S \le Np/2\}}] &\le (E|U|^4)^{1/2} P(S \le Np/2)^{1/2}\\
    &\le C N^{a + 1} p^{a + 1} \exp\left(-\frac {1 - \log 2} 4 \, Np\right)\\
    &\le \frac C {N^{a + 1} p^{a + 1}}
  \end{align*}
  Combining the two gives $E|U|^2 \le C/N^{a + 1} p^{a + 1}$.

  We therefore have
  \begin{multline*}
    \|R\|_{L^2(S \times \Om)}^2
    = E \int |R|^2 \, d\nu
    = E \int \sum_{j, k} U_{j, k}^2 1_{A_{j, k}} \, d\nu
    = \sum_{j, k} (E U_{j, k}^2) \nu(A_{j, k})\\
    \le \frac C {N^{a + 1}} \sum_{j, k} e^{(a + 1)Xf(y_{j, k})} \, \nu(A_{j, k})
    = \frac C {N^{a + 1}} \int e^{(a + 1)X_{n, m}f} \, d\nu
    \le \frac C {N^{a + 1}}.
  \end{multline*}
  Taking square roots finishes the proof.
\end{proof}

With Lemma \ref{L:remEst} in hand, we can now prove our functional law of large
numbers.

\begin{proof}[Proof of Theorem \ref{T:LLN}]
  We will use the Taylor expansion,
  \begin{equation*}
    \log x = \log c + \sum_{r = 1}^a {
      (-1)^{r - 1} \frac {(x - c)^r} {r c^r}
    } + (-1)^a \int_c^x \frac {(x - u)^a} {u^{a + 1}} \, du.
  \end{equation*}
  For notational simplicity, let $S = S^{y_{j, k}}f$ and $p = p_{j, k}$. We then
  have
  \begin{equation*}
    \log \left( \frac {S + 1} N \right) - \log \left( p + \frac 1 N \right)
      = \sum_{r = 1}^a {
        (-1)^{r - 1} \frac {(S - Np)^r} {r (Np + 1)^r}
      } + (-1)^a \int_{Np}^S \frac {(S - t)^a} {(t + 1)^{a + 1}} \, dt
  \end{equation*}
  and
  \begin{equation*}
    \log \left( p + \frac 1 N \right) - \log p = \sum_{r = 1}^b {
      (-1)^{r - 1} \frac 1 {r N^r p^r}
    } + (-1)^b \int_{Np - 1}^{Np} \frac {(Np - t)^b} {(t + 1)^{b + 1}} \, dt,
  \end{equation*}
  where in the integral remainders, we have made the substitution $t = Nu - 1$.
  Combining these, we have
  \begin{multline}\label{Taylor}
    - \log \left( \frac {S + 1} N \right) = - \log p + \sum_{r = 1}^a {
      (-1)^r \frac {(S - Np)^r} {r (Np + 1)^r}
    } + \sum_{r = 1}^b {
      (-1)^r \frac 1 {r N^r p^r}
    }\\
    + (-1)^{b + 1} \int_{Np - 1}^{Np} {
      \frac {(Np - t)^b} {(t + 1)^{b + 1}}
    } \, dt + (-1)^{a + 1} \int_{Np}^S {
      \frac {(S - t)^a} {(t + 1)^{a + 1}}
    } \, dt
  \end{multline}
  Taking $a = 1$ and $b = 0$ in \eqref{Taylor}, we have
  \begin{equation*}
    - \log \left( \frac {S + 1} N \right) = - \log p + \frac {Np - S} {Np + 1}
      - \int_{Np - 1}^{Np} {
        \frac 1 {t + 1}
      } \, dt + \int_{Np}^S {
        \frac {S - t} {(t + 1)^2}
      } \, dt
  \end{equation*}
  Substituting this into \eqref{Ydef} gives
  \begin{equation*}
    Y_{n, m, N}f = X_{n, m}f + \frac1{\sqrt{nmN}} \, W_{n, m, N}f - R_1 + R_2,
  \end{equation*}
  where $W_{n, m, N}f$ is given by \eqref{discWN},
  \begin{align*}
    R_1 &= \sum_{j, k} \int_{Np_{j, k} - 1}^{Np_{j, k}} {
      \frac 1 {t + 1}
    } \, dt \, 1_{A_{j, k}}, \text{ and}\\
    R_2 &= \sum_{j, k} \int_{Np_{j, k}}^{S^{y_{j, k}}f} {
      \frac {S^{y_{j, k}}f - t} {(t + 1)^2}
    } \, dt \, 1_{A_{j, k}}.
  \end{align*}
  For the first term, we have
  \begin{align*}
    \left\|
      \frac1{\sqrt{nmN}} \, W_{n, m, N}f
    \right\|_{L^2(S \times \Om)}^2 &= E\left\|
      \frac 1 {\sqrt{nmN}} \, W_{n, m, N}f
    \right\|_2^2\\
    &= E \int \bigg|
      \sum_{j, k} \frac {Np - S} {Np + 1} \, 1_{A_{j, k}}
    \bigg|^2 \, d\nu\\
    &= E \int \sum_{j, k} {
      \frac {|Np - S|^2} {(Np + 1)^2} \, 1_{A_{j, k}}
    } \, d\nu.
  \end{align*}
  Since $S \sim \Pois(Np)$, this gives
  \begin{align*}
    \left\|
      \frac 1 {\sqrt{nmN}} \, W_{n, m, N}f
    \right\|_{L^2(S \times \Om)}^2 &= \sum_{j, k} {
      \frac {Np} {(Np + 1)^2} \, \nu(A_{j, k})
    }\\
    &\le \frac 1 N \sum_{j, k} {
      e^{Xf(y_{j, k})} \nu(A_{j, k})
    }\\
    &= \frac 1 N \int e^{X_{n, m}f} \, d\nu.
  \end{align*}
  Therefore,
  \begin{equation*}
    \left\|
      \frac 1 {\sqrt{nmN}} \, W_{n, m, N}f
    \right\|_{L^2(S \times \Om)} \le \frac C {\sqrt{N}}.
  \end{equation*}
  For the first remainder term, we have
  \begin{equation*}
    \|R_1\|_{L^2(S \times \Om)}^2 = \|R_1\|_2^2 \le \sum_{j, k} {
      \frac 1 {N^2 p_{j, k}^2} \, \nu(A_{j, k})
    } = \frac 1 {N^2} \int e^{2X_{n, m}f} \, d\nu,
  \end{equation*}
  so that $\|R_1\|_{L^2(S \times \Om)} \le C/N$. For the second remainder term,
  we apply Lemma \ref{L:remEst} with $a = 1$ to obtain $\|R_2\|_{L^2(S \times
  \Om)} \le C/N$.
\end{proof}

\section{Asymptotics of the Berry-Esseen parameters}\label{S:BEparams}

The proof of our functional central limit is built upon the Berry-Esseen
inequality \eqref{BEraw}. To make use of this, we must understand the
asymptotics of the two parameters, $\si_{n, m, N}$ and $L_{n, m, N}$, given by
\eqref{Wvar} and \eqref{Wskew}. These asymptotics are established in
Propositions \ref{P:variance} and \ref{P:skew} below, which show that $\si_{n,
m, N} \approx \|\pi e^{Xf/2} g\|_2$ and $L_{n, m, N} \approx C/\sqrt{nm}$. Both
of these results are consequences of the following lemma.

\begin{lemma}\label{L:weakSum}
  Let $\ka \in (0, \infty)$ and let $h: [0, \infty) \to [0, \infty)$ be a
  bounded function that satisfies $h(\la) = 1 + O(\la^{-\ka})$ as $\la \to
  \infty$. Let $r \ge 1$ and $s > 0$. Then there exists a constant $C$ that
  depends only on $r$, $s$, $f$, $g$, $h$, and $\ka$ such that
  \begin{equation*}
    \bigg|
      \sum_{j, k} {
        (nm)^{r - 1} h(Np_{j, k}) p_{j, k}^{-s} |\ga(A_{j, k})|^r
      } - \pi^{2r - 2} \int e^{sXf} |g|^r \, d\nu
    \bigg| \le C \left(
      \frac 1 n + \frac 1 m + \frac 1 {N^\ka}
    \right).
  \end{equation*}
\end{lemma}

\begin{proof}
  Let us write
  \begin{equation}\label{weakSum1}
    \sum_{j, k} {
      (nm)^{r - 1} h(Np_{j, k}) p_{j, k}^{-s} |\ga(A_{j, k})|^r
    } - \pi^{2r - 2} \int e^{sXf} |g|^r \, d\nu = \ep_1 + \ep_2 + \ep_3,
  \end{equation}
  where
  \begin{align*}
    \ep_1 &= \sum_{j, k} (nm)^{r - 1} h(Np_{j, k}) p_{j, k}^{-s} (
      |\ga(A_{j, k})|^r - |g(y_{j, k})|^r \nu(A_{j, k})^r
    ),\\
    \ep_2 &= \pi^{2r - 2} \sum_{j, k} h(Np_{j, k})\int_{A_{j, k}} (
      e^{sXf(y_{j, k})} |g(y_{j, k})|^r - e^{sXf} |g|^r
    ) \, d\nu,\\
    \ep_3 &= \pi^{2r - 2} \sum_{j, k} \bigg(
      \int_{A_{j, k}} e^{sXf} |g|^r \, d\nu
    \bigg) (h(Np_{j, k}) - 1).
  \end{align*}
  For $r \ge 1$ and $a, b > 0$, the mean value theorem implies $|a^r - b^r| \le
  r (a \vee b)^{r - 1} |a - b|$. Let us apply this with $a = |\int_{A_{j, k}} g
  \, d\nu|$ and $b = |g(y_{j, k})| \nu(A_{j, k})$. Note that $a \vee b \le
  \|g\|_\infty \nu(A_{j, k}) = \|g\|_\infty \pi^2/nm$. Since $g$ is Lipschitz
  and the diameter of $A_{j, k}$ is bounded above by $\De s_j + 2\pi/m$, we have
  \begin{align*}
    \left|
      |\ga(A_{j, k})|^r - |g(y_{j, k})|^r \nu(A_{j, k})^r
    \right| &\le \frac C {(nm)^{r - 1}} \left|
      |\ga(A_{j, k})| - |g(y_{j, k})| \nu(A_{j, k})
    \right|\\
    &\le \frac C {(nm)^{r - 1}} \left|
      \ga(A_{j, k}) - g(y_{j, k}) \nu(A_{j, k})
    \right|\\
    &\le \frac C {(nm)^{r - 1}} \int_{A_{j, k}} |g - g(y_{j, k})| \, d\nu\\
    &\le \frac C {(nm)^r} \left(
      \De s_j + \frac {2\pi} m
    \right).
  \end{align*}
  This gives
  \begin{equation*}
    |\ep_1| \le \frac C {(nm)^r} \sum_{j, k} (nm)^{r - 1} \left(
      \De s_j + \frac {2\pi} m
    \right) = C \left( \frac {m + 2\pi n} {nm} \right) \le C \left(
      \frac 1 n + \frac 1 m
    \right).
  \end{equation*}
  Next, note that $e^{sXf}|g|^r$ is Lipschitz. Hence,
  \begin{equation*}
    |e^{sXf(y_{j, k})} |g(y_{j, k})|^r - e^{sXf(y)} |g(y)|^r|
      \le C \left(\De s_j + \frac {2\pi} m\right).
  \end{equation*}
  Therefore,
  \begin{equation*}
    |\ep_2| \le \frac C {nm} \sum_{j, k} \left(
      \De s_j + \frac {2\pi} m
    \right) = C \left( \frac {m + 2\pi n} {nm} \right) \le C \left(
      \frac 1 n + \frac 1 m
    \right).
  \end{equation*}
  Finally, using $p_{j, k} \ge e^{-\|Xf\|_\infty}$, we have
  \begin{equation*}
    |\ep_3| = \pi^{2r - 2} \sum_{j, k} \bigg(
      \int_{A_{j, k}} e^{sXf} |g|^r \, d\nu
    \bigg) |h(Np_{j, k}) - 1| \le \pi^{2r - 2} \int {
      e^{sXf} |g|^r \, d\nu
    } \, \frac {e^{\ka \|Xf\|_\infty}} {N^\ka} = \frac C {N^\ka}.
  \end{equation*}
  Putting these estimates into \eqref{weakSum1} completes the proof.
\end{proof}

\begin{prop}\label{P:variance}
  There exists a constant $C$ that depends only on $f$ and $g$ such that
  \begin{equation*}
    |\si_{n, m, N}^2 - \|\pi e^{Xf/2} g\|_2^2| \le C \left(
      \frac 1 n + \frac 1 m + \frac 1 N
    \right).
  \end{equation*}
\end{prop}

\begin{proof}
  Let $h(\la) = \la^2/(\la + 1)^2$. Since $h(\la) = 1 + O(\la^{-1})$, the result
  follows from Lemma \ref{L:weakSum} with $r = 2$ and $s = 1$.
\end{proof}

\begin{prop}\label{P:skew}
  There exists a constant $C$ that depends only on the functions $f$ and $g$
  such that $L_{n, m, N} \le C/\sqrt{nm}$.
\end{prop}

\begin{proof}
  By Proposition \ref{P:variance} and Lemma \ref{L:PoisMom},
  \begin{align*}
    L_{n, m, N} &\le C \sum_{j, k} \frac {
      (nm)^{3/2} N^3 p_{j, k}^{3/2}
    } {
      (Np_{j, k} + 1)^3
    } \, |\ga(A_{j, k})|^3\\
    &= C(nm)^{-1/2} \sum_{j, k} {
      (nm)^2 h(Np_{j, k}) p_{j, k}^{-3/2} \, |\ga(A_{j, k})|^3
    },
  \end{align*}
  where $h(\la) = \la^3/(\la + 1)^3$. Since $h(\la) = 1 + O(\la^{-1})$, the
  result follows from Lemma \ref{L:weakSum} with $r = 3$ and $s = 3/2$.
\end{proof}

\section{Proving the functional central limit theorem}\label{S:CLTpf}

For our proof of the functional central limit, Theorem \ref{T:CLT}, we return to
the Taylor expansion, \eqref{Taylor}. For the law of large numbers, we took $a =
1$ and $b = 0$. For the central limit theorem, we take larger values of $a$ and
$b$ to make the remainders smaller. The terms in the second sum on the
right-hand side of \eqref{Taylor} will become correction terms in the
normalization used in the central limit theorem. The first term in the first sum
will become $W_{n, m, N}f$. But the other terms in the first sum are
problematic. The $r$-th such term has an $L^2$ norm whose magnitude is
$N^{-r/2}$, as demonstrated by a simpler version of the proof of Lemma
\ref{L:remEst}. No matter how large $a$ is, we will still have the term
corresponding to $r = 2$, with magnitude $N^{-1}$. Combined with the
$\sqrt{nmN}$ multiplier in Theorem \ref{T:CLT}, it seems we can never do better
than requiring $nm/N \to 0$. This, of course, is useless, as explained in the
introduction.

To deal with this, we will center the terms in the first sum so that they have
mean zero. By itself, this would not be enough, since even with the centering,
the terms still have $L^2$ norm of size $N^{-r/2}$. But by measuring these terms
in the weak sense rather than the $L^2$ sense, we find that they have size
$(nm)^{-1/2}N^{-r/2}$. (See Lemma \ref{L:weakRem} below.) Now when we combine
this with the $\sqrt{nmN}$ multiplier in Theorem \ref{T:CLT}, even for $r = 2$,
we get convergence to zero whenever $n, m, N \to \infty$, without any additional
condition. The additional condition is needed only for the remainders that arise
from the integral expressions in \eqref{Taylor}.

We begin with the following lemma, which derives the new Taylor expansion for
$Y_{n, m, N}f$ that arises after the centering described above. This centering
produces a second set of correction terms to be used in the normalization of
Theorem \ref{T:CLT}.

\begin{lemma}\label{L:Taylor+}
  For any nonnegative integers $a$ and $b$, there exists a constant $C$ that
  depends only on $a$, $b$, and $f$ such that
  \begin{multline}\label{Taylor+}
    Y_{n, m, N}f = X_{n, m}f + \sum_{r = 1}^b {
      \frac {(-1)^r} {rN^r} \, e^{rX_{n, m}f}
    } + \sum_{r = 2}^a \frac {
      (-1)^r \mu_r(N e^{-X_{n, m}f})
    } {
      r (N e^{-X_{n, m}f} + 1)^r
    }\\
    + \sum_{r = 1}^a \frac {(-1)^r} r \sum_{j, k} \frac {
      (S^{y_{j, k}}f - Np_{j, k})^r - \mu_r(Np_{j, k})
    } {
      (Np_{j, k} + 1)^r
    } \, 1_{A_{j, k}} + R,
  \end{multline}
  where $\|R\|_{L^2(S \times \Om)} \le C(N^{-(b + 1)} + N^{-(a + 1)/2})$.
\end{lemma}

\begin{proof}
  By \eqref{Taylor}, we have
  \begin{multline*}
    Y_{n, m, N}f = X_{n, m}f + \sum_{r = 1}^a \frac {(-1)^r} r \sum_{j, k} {
      \frac {
        (S^{y_{j, k}}f - Np_{j, k})^r - \mu_r(Np_{j, k})
      } {
        (Np_{j, k} + 1)^r
      } \, 1_{A_{j, k}}
    }\\
    + \sum_{r = 1}^a \frac {
      (-1)^r \mu_r(N e^{-X_{n, m}f})
    } {
      r (N e^{-X_{n, m}f} + 1)^r
    } + \sum_{r = 1}^b \frac {(-1)^r} {rN^r} \sum_{j, k} {
      e^{rXf(y_{j, k})} \, 1_{A_{j, k}}
    }\\
    + (-1)^{b + 1} R_1 + (-1)^{a + 1} R_2,
  \end{multline*}
  where
  \begin{align*}
    R_1 &= \sum_{j, k} \int_{Np_{j, k} - 1}^{Np_{j, k}} {
      \frac {(Np_{j, k} - t)^b} {(t + 1)^{b + 1}}
    } \, dt \, 1_{A_{j, k}} \text{ and}\\
    R_2 &= \sum_{j, k} \int_{Np_{j, k}}^{S^{y_{j, k}}f} {
      \frac {(S^{y_{j, k}}f - t)^a} {(t + 1)^{a + 1}}
    } \, dt \, 1_{A_{j, k}}.
  \end{align*}
  Since $\mu_1 \equiv 0$, this verifies \eqref{Taylor+} with $R = (-1)^{b + 1}
  R_1 + (-1)^{a + 1} R_2$.

  To estimate $R_1$, note that
  \begin{equation*}
    \|R_1\|_{L^2(S \times \Om)}^2 = \|R_1\|_2^2 \le \sum_{j, k} {
      \frac 1 {N^{2b + 2} p_{j, k}^{2b + 2}} \, \nu(A_{j, k})
    } = \frac 1 {N^{2b + 2}} \int e^{(2b + 2)X_{n, m}f} \, d\nu,
  \end{equation*}
  so that $\|R_1\|_{L^2(S \times \Om)} \le C/N^{b + 1}$. To estimate $R_2$, we
  may apply Lemma \ref{L:remEst} in order to obtain $\|R_2\|_{L^2(S \times \Om)}
  \le C/N^{(a + 1)/2}$.
\end{proof}

\begin{lemma}\label{L:weakRem}
  Let $r \in \bN$ and $g \in C_c^\infty(Z)$. Define
  \begin{equation*}
    R = \sum_{j, k} \frac {
      (S^{y_{j, k}}f - Np_{j, k})^r - \mu_r(Np_{j, k})
    } {
      (Np_{j, k} + 1)^r
    } \, 1_{A_{j, k}}.
  \end{equation*}
  Then there exists a constant $C$ that depends only on $r$, $f$, and $g$ such
  that $E|\ang{R, g}|^2 \le C/nmN^r$.
\end{lemma}

\begin{proof}
  Recall that $d\ga = g \, d\nu$, so that
  \begin{equation*}
    \ang{R, g} = \sum_{j, k} \frac {
      (S^{y_{j, k}}f - Np_{j, k})^r - \mu_r(Np_{j, k})
    } {
      (Np_{j, k} + 1)^r
    } \, \ga(A_{j, k}).
  \end{equation*}
  Since this is a sum of independent, mean zero random variables, we have
  \begin{equation*}
    E|\ang{R, g}|^2 = \sum_{j, k} \frac {
      E|(S^{y_{j, k}}f - Np_{j, k})^r - \mu_r(Np_{j, k})|^2
    } {
      (Np_{j, k} + 1)^{2r}
    } \, \ga(A_{j, k})^2.
  \end{equation*}
  Using $\nu(A_{j, k}) = \pi^2/nm$, $p_{j, k} \ge e^{-\|Xf\|_\infty}$, the fact
  that $\mu_r$ is a polynomial of degree $\flr{r/2}$, and applying Lemma
  \ref{L:PoisMom}, we obtain
  \begin{equation*}
    E|\ang{R, g}|^2 \le C \sum_{j, k} \frac 1 {N^r} \, \frac 1 {n^2 m^2}
      = \frac C {nmN^r},
  \end{equation*}
  which is what we wanted to prove.
\end{proof}

\begin{lemma}\label{L:normEst}
  Let $a, b > 0$ and let $\Phi$ be the standard normal distribution function.
  Then
  \begin{align}
    |\Phi(ax) - \Phi(bx)|
      &\le \frac 1 {\sqrt{2\pi e}} \, \frac {|b - a|} {a \wedge b},
      \label{normEst1} \\
    |\Phi(x/a) - \Phi(x/b)|
      &\le \frac 1 {\sqrt{2\pi e}} \, \frac {|b - a|} {a \wedge b},
      \label{normEst2}
  \end{align}
  for all $x \in \bR$.
\end{lemma}

\begin{proof}
  Without loss of generality, we may assume that $a < b$. Then
  \begin{equation*}
    |\Phi(ax) - \Phi(bx)|
      = \frac 1 {\sqrt{2\pi}} \int_{a|x|}^{b|x|} e^{-u^2/2} \, du
      \le \frac 1 {\sqrt{2\pi}} \, (b - a) |x| e^{-a^2 x^2/2}.
  \end{equation*}
  Since the maximum of the right-hand side occurs when $|x| = 1/a$, this gives
  \eqref{normEst1}. Since $0 < 1/b < 1/a$, applying \eqref{normEst1} gives
  \begin{equation*}
    |\Phi(x/a) - \Phi(x/b)|
      \le \frac 1 {\sqrt{2\pi e}} \, \frac {a^{-1} - b^{-1}} {b^{-1}}
      = \frac 1 {\sqrt{2\pi e}} \, \frac {b - a} a,
  \end{equation*}
  which proves \eqref{normEst2}.
\end{proof}

\begin{proof}[Proof of Theorem \ref{T:CLT}]
  By L\'evy's continuity theorem, it suffices to show that $\ang{Z_{n, m, N}f, g}
  \To \ang{\pi e^{Xf/2} W, g}$ for all $g \in C_c^\infty(Z)$. The white noise
  $W$ is characterized by $\ang{W, g} =_d \|g\|_2 \eta$ for all $g \in
  C_c^\infty(Z)$, where $\eta \sim N(0, 1)$. Hence, $\ang{\pi e^{Xf/2} W, g} =_d
  \|\pi e^{Xf/2} g\|_2 \eta$. We must therefore show that $\ang{Z_{n, m, N}f, g}
  \To \|\pi e^{Xf/2} g\|_2 \eta$ for all $g \in C_c^\infty(Z)$.

  Fix $g \in C_c^\infty$. In the following proof, $C$ will denote a constant
  that depends only on $a$, $b$, $f$, and $g$. Now, if $F$ is an $L^2(Z)$-valued
  random variable, then $E|\ang{F, g}|^2 \le \|g\|_2 E\|F\|_2^2$. Hence, by
  \eqref{discWN} and Lemmas \ref{L:Taylor+} and \ref{L:weakRem}, we may write
  \begin{equation}\label{WN+err}
    Z_{n, m, N}f = W_{n, m, N}f + \be_{n, m, N}f,
  \end{equation}
  where
  \begin{align*}
    E|\ang{\be_{n, m, N}f, g}|^2 &\le C nmN \bigg(
      \sum_{r = 2}^a \frac 1 {nmN^r} + \frac 1 {N^{2b + 2}} + \frac 1 {N^{a + 1}}
    \bigg)\\
    &= C \left(
      \frac 1 N + \frac {nm} {N^{2b + 1}} + \frac {nm} {N^a}
    \right)\\
    &\le C \left(
      \frac 1 N + \frac {nm} {N^\ka}
    \right).
  \end{align*}
  Define $\ep = (N^{-1} + nmN^{-\ka})^{1/3}$, so that
  \begin{equation*}
    P(|\ang{\be_{n, m, N}f, g}| > \ep)
        \le \ep^{-2} E|\ang{\be_{n, m, N}f, g}|^2
        \le C \ep.
  \end{equation*}
  Using \eqref{BEraw}, Lemma \ref{L:normEst}, and Propositions \ref{P:variance}
  and \ref{P:skew}, we have
  \begin{align*}
    |P(
      \langle W_{n, m, N} f, g \rangle \le x
    ) &- P(
      \|\pi e^{Xf/2} g\|_2 \eta \le x
    )|\\
    &\le |P(
      \si_{n, m, N}^{-1} \ang{W_{n, m, N}f, g} \le \si_{n, m, N}^{-1} x
    ) - \Phi(\si_{n, m, N}^{-1} x))|\\
    &\hspace{2.25in} + |
      \Phi(\si_{n, m, N}^{-1} x) - \Phi(\|\pi e^{Xf/2} g\|_2^{-1} x)
    |\\
    &\le C \left(
      \frac 1 {\sqrt{nm}} + \frac 1 n + \frac 1 m + \frac 1 N
    \right),
  \end{align*}
  for all $x \in \bR$. We may therefore write,
  \begin{align*}
    P(\ang{Z_{n, m, N}f, g} \le x)
      &= P(\ang{W_{n, m, N}f, g} + \ang{\be_{n, m, N}f, g} \le x)\\
    &\le P(\ang{W_{n, m, N}f, g} \le x + \ep)
      + P(|\ang{\be_{n, m, N}f, g}| > \ep)\\
    &\le P(\|\pi e^{Xf/2} g\|_2 \eta \le x + \ep) + C \left(
      \frac 1 {\sqrt{nm}} + \frac 1 n + \frac 1 m + \frac 1 N + \ep
    \right)\\
    &\le P(\|\pi e^{Xf/2} g\|_2 \eta \le x) + C \left(
      \frac 1 {\sqrt{nm}} + \frac 1 n + \frac 1 m + \frac 1 N + \ep
    \right),
  \end{align*}
  where in the last inequality we have once again used Lemma \ref{L:normEst}.
  Similarly,
  \begin{align*}
    P(\ang{Z_{n, m, N}f, g} \le x)
      &\ge P(\ang{W_{n, m, N}f, g} \le x - \ep)
      - P(|\ang{\be_{n, m, N}f, g}| > \ep)\\
    &\ge P(\|\pi e^{Xf/2} g\|_2 \eta \le x) - C \left(
      \frac 1 {\sqrt{nm}} + \frac 1 n + \frac 1 m + \frac 1 N + \ep
    \right).
  \end{align*}
  By the definition of $\ep$, this proves \eqref{BE}. It follows that
  $\ang{Z_{n, m, N}f, g} \To \|\pi e^{Xf/2} g\|_2 \eta$ whenever $n, m, N \to
  \infty$ and $nm/N^\ka \to 0$.
\end{proof}

\begin{rmk}
  The term $\be_{n, m, N}f$ in \eqref{WN+err} can be written explicitly by
  tracing through the definition of $Z_{n, m, N}f$ in the statement of Theorem
  \ref{T:CLT} and comparing this to both the definition of $W_{n, m, N}f$ in
  \eqref{discWN} and the decomposition \eqref{Taylor+}. Doing this, we find that
  \begin{multline*}
    \be_{n, m, N}f = \sqrt{nmN} \, \bigg(
      \sum_{r = 2}^a \frac {(-1)^r} r \sum_{j, k} \frac {
        (S^{y_{j, k}}f - Np_{j, k})^r - \mu_r(Np_{j, k})
      } {
        (Np_{j, k} + 1)^r
      } \, 1_{A_{j, k}}\\
      + (-1)^{b + 1} \sum_{j, k} {
        \int_{Np_{j, k} - 1}^{Np_{j, k}} {
          \frac {(Np_{j, k} - t)^b} {(t + 1)^{b + 1}}
        } \, dt \, 1_{A_{j, k}}
      } + (-1)^{a + 1} \sum_{j, k} {
        \int_{Np_{j, k}}^{S^{y_{j, k}}f} {
          \frac {(S^{y_{j, k}}f - t)^a} {(t + 1)^{a + 1}}
        } \, dt \, 1_{A_{j, k}}
      }
    \bigg).
  \end{multline*}
  Comparing this to \eqref{discWN}, the terms $W_{n, m, N}f$ and $\be_{n, m,
  N}f$ in \eqref{WN+err} evidently have a complicated structure of dependency.
  This is not relevant, however, to our proof of Theorem \ref{T:CLT}, since
  $\be_{n, m, N}f$ is of a smaller order of magnitude than $W_{n, m, N}f$. In
  short, Theorem \ref{T:CLT} holds because $W_{n, m, N}f$ converges in
  distribution to a white noise, whereas $\be_{n, m, N}f$ converges to zero in
  probability. Nonetheless, one should not be misled into thinking that $Z_{n,
  m, N}f$ looks like a white noise plus a small, independent perturbation.
  Rather, it looks like a white noise plus a small, dependent perturbation.
\end{rmk}

\begin{proof}[Proof of Corollary \ref{C:CLT3}]
  We apply Theorem \ref{T:CLT} with $a = 3$ and $b = 1$, so that $\ka = 3$. In
  this case,
  \begin{equation*}
    Z_{n, m, N}f = \sqrt{nmN} \, \bigg(
      Y_{n, m, N}f - X_{n, m}f + \frac 1 N \, e^{X_{n, m}f}
      - \frac 1 2 \, h_2(N e^{-X_{n, m}f})
      + \frac 1 3 \, h_3(N e^{-X_{n, m}f})
    \bigg),
  \end{equation*}
  where
  \begin{align}
    h_2(\la) &= {
      \frac \la {(\la + 1)^2} = \frac 1 \la + O(\la^{-2})
    } \text{ and} \label{h2asymp} \\
    h_3(\la) &= \frac \la {(\la + 1)^3} = O(\la^{-2}) \notag
  \end{align}
  as $\la \to \infty$. Hence, there is a constant $C$ that depends only on $f$
  such that
  \begin{align*}
    \left|
      h_2(N e^{-X_{n, m}f}) - \frac 1 N \, e^{X_{n, m}f}
    \right| &\le \frac C {N^2} \text{ and}\\
    |h_3(N e^{-X_{n, m}f})| &\le \frac C {N^2}.
  \end{align*}
  We therefore have
  \begin{equation*}
    Z_{n, m, N}f = \sqrt{nmN} \, \bigg(
      Y_{n, m, N}f - X_{n, m}f + \frac 1 {2N} \, e^{X_{n, m}f} + R
    \bigg),
  \end{equation*}
  where $\|\sqrt{nmN} \, R\|_2^2 \le CnmN/N^4 = Cnm/N^3$.
\end{proof}

\begin{proof}[Proof of Corollary \ref{C:CLT5}]
  We apply Theorem \ref{T:CLT} with $a = 5$ and $b = 2$, so that $\ka = 5$. In
  this case,
  \begin{multline*}
    Z_{n, m, N}f = \sqrt{nmN} \, \bigg(
      Y_{n, m, N}f - X_{n, m}f + \frac 1 N \, e^{X_{n, m}f}
      - \frac 1 {2N^2} \, e^{2X_{n, m}f}\\
    - \frac 1 2 \, h_2(N e^{-X_{n, m}f})
      + \frac 1 3 \, h_3(N e^{-X_{n, m}f})
      - \frac 1 4 \, h_4(N e^{-X_{n, m}f})
      + \frac 1 5 \, h_5(N e^{-X_{n, m}f})
    \bigg),
  \end{multline*}
  where
  \begin{align*}
    h_2(\la) &= \frac \la {(\la + 1)^2}
      = \frac 1 \la - \frac 2 {\la^2} + O(\la^{-3}),\\
    h_3(\la) &= \frac \la {(\la + 1)^3}
      = \frac 1 {\la^2} + O(\la^{-3}),\\
    h_4(\la) &= \frac {3\la^2 + \la} {(\la + 1)^4}
      = \frac 3 {\la^2} + O(\la^{-3}), \text{ and}\\
    h_5(\la) &= \frac {10\la^2 + \la} {(\la + 1)^5} = O(\la^{-3})
  \end{align*}
  as $\la \to \infty$. Hence, there is a constant $C$ that depends only on $f$
  such that
  \begin{align*}
    \left|
      h_2(N e^{-X_{n, m}f}) - \frac 1 N \, e^{X_{n, m}f}
        + \frac 2 {N^2} \, e^{2X_{n, m}f}
    \right| &\le \frac C {N^3},\\
    \left|
      h_3(N e^{-X_{n, m}f}) - \frac 1 {N^2} \, e^{2X_{n, m}f}
    \right| &\le \frac C {N^3},\\
    \left|
      h_4(N e^{-X_{n, m}f}) - \frac 3 {N^2} \, e^{2X_{n, m}f}
    \right| &\le \frac C {N^3}, \text{ and}\\
    |h_5(N e^{-X_{n, m}f})| &\le \frac C {N^3}.
  \end{align*}
  We therefore have
  \begin{align*}
    Z_{n, m, N}f &= \sqrt{nmN} \, \bigg({
      Y_{n, m, N}f - X_{n, m}f
      + \frac 1 N \, e^{X_{n, m}f}
      - \frac 1 {2N^2} \, e^{2X_{n, m}f}
    }\\
    &\qquad - \frac 1 2 \bigg(
      h_2(N e^{-X_{n, m}f})
      - \frac 1 N \, e^{X_{n, m}f} + \frac 2 {N^2} \, e^{2X_{n, m}f}
    \bigg)\\
    &\qquad + \frac 1 3 \bigg(
      h_3(N e^{-X_{n, m}f}) - \frac 1 {N^2} \, e^{2X_{n, m}f}
    \bigg) - \frac 1 4 \bigg(
      h_4(N e^{-X_{n, m}f}) - \frac 3 {N^2} \, e^{2X_{n, m}f}
    \bigg)\\
    &\qquad + \frac 1 5 \, h_5(N e^{-X_{n, m}f})\\
    &\qquad {
      - \frac 1 {2N} \, e^{X_{n, m}f}
      + \frac 1 {N^2} \, e^{2X_{n, m}f}
      + \frac 1 {3N^2} \, e^{2X_{n, m}f}
      - \frac 3 {4N^2} \, e^{2X_{n, m}f}
    }\bigg)\\
    &= \sqrt{nmN} \, \bigg(
      Y_{n, m, N}f - X_{n, m}f
      + \frac 1 {2N} \, e^{X_{n, m}f}
      + \frac 1 {12N^2} \, e^{2X_{n, m}f} + R
    \bigg)
  \end{align*}
  where $\|\sqrt{nmN} \, R\|_2^2 \le CnmN/N^6 = Cnm/N^5$.
\end{proof}

In the introduction, we noted that if $n$ and $m$ have the same order of
magnitude, then we must take $N \approx n^2$ to use the law of large numbers.
But in our central limit theorem, Theorem \ref{T:CLT}, we need $N \gg
n^{2/\ka}$. For this to be an improvement on the law of large numbers, we must
take $\ka \ge 2$. The smallest $a$ and $b$ that produce $\ka = 2$ are $a = 2$
and $b = 1$. In this case, in the notation of the proof of Corollary
\ref{C:CLT3}, we have
\begin{equation*}
  Z_{n, m, N}f = \sqrt{nmN} \, \bigg(
    Y_{n, m, N}f - X_{n, m}f + \frac 1 N \, e^{X_{n, m}f}
    - \frac 1 2 \, h_2(N e^{-X_{n, m}f})
  \bigg),
\end{equation*}
without the $h_3$ term. Since $h_2(\la) = O(\la^{-1})$, we must utilize
\eqref{h2asymp} in order to do better than $nm/N \to 0$. But when we do this, we
merely conclude that \eqref{CLT3} holds whenever $nm/N^2 \to 0$. In other words,
we gain no simplification to \eqref{CLT3} by weakening the condition from
$nm/N^3 \to 0$ to $nm/N^2 \to 0$. Hence, the instance of Theorem \ref{T:CLT}
that corresponds to $\ka = 2$ is redundant when compared to $\ka = 3$.
Similarly, we gain no simplification to \eqref{CLT5} by taking $\ka = 4$ instead
of $\ka = 5$. In fact, by examining the structures of the proofs of Corollaries
\ref{C:CLT3} and \ref{C:CLT5}, it is clear that the only non-redundant instances
of Theorem \ref{T:CLT} are those in which $\ka$ is odd.

\section{Alternative normalization methods}\label{S:altNorm}

Let us define
\begin{equation*}
  \wh Y_{n, m, N}f = - \sum_{j, k} {
    \log \left( \frac{S^{y_{j, k}}f \vee 1} N \right) 1_{A_{j, k}}
  },
\end{equation*}
For notational simplicity, let $S = S^{y_{j, k}}f$ and $p = p_{j, k}$. Note that
$S \vee 1 = S + 1_{\{S = 0\}}$, so that $E[S \vee 1] = Np + e^{-Np}$. As in the
proof of Theorem \ref{T:LLN}, if we let $\ze = \ze_{j, k} = 1_{\{S = 0\}} -
e^{-Np}$, then
\begin{equation*}
  \log \left( \frac {S \vee 1} N \right) - \log \left(
    p + \frac 1 N \, e^{-Np}
  \right) = \sum_{r = 1}^a {
      (-1)^{r - 1} \frac {(S - Np + \ze)^r} {r (Np + e^{-Np})^r}
    } + (-1)^a \int_{Np}^{S + \ze} {
      \frac {(S + \ze - t)^a} {(t + e^{-Np})^{a + 1}} \, dt
    }
\end{equation*}
and
\begin{equation*}
  \log \left(
    p + \frac 1 N \, e^{-Np}
  \right) - \log p = \int_{Np - e^{-Np}}^{Np} {
    \frac 1 {t + e^{-Np}} \, dt
  },
\end{equation*}
where in the second expansion, we have taken $b = 0$. Let
\begin{equation*}
  \wh \mu_r(\la) = E[(S - \la + 1_{\{S = 0\}} - e^{-\la})^r],
\end{equation*}
where $S \sim \Pois(\la)$. Then for any $a \ge 1$, we have
\begin{multline*}
  \wh Y_{n, m, N}f = X_{n, m}f + \frac 1 {\sqrt{nmN}} \, \wh W_{n, m, N}f
  + \sum_{r = 2}^a \frac {
    (-1)^r \wh \mu_r(N e^{-X_{n, m}f})
  } {
    r (N e^{-X_{n, m}f} + e^{-Np_{j, k}})^r
  }\\
  + \sum_{r = 2}^a \frac {(-1)^r} r \sum_{j, k} \frac {
    (S^{y_{j, k}}f - Np_{j, k} + \ze_{j, k})^r - \wh \mu_r(Np_{j, k})
  } {
    (Np_{j, k} + e^{-Np_{j, k}})^r
  } \, 1_{A_{j, k}} - \wh R_1 + (-1)^{a + 1} \wh R_2,
\end{multline*}
where
\begin{align*}
  \wh W_{n, m, N}f &= \sqrt{nmN} \, \sum_{j, k} \frac{
    Np_{j, k} - S^{y_{j, k}}f - \ze_{j, k}
  }{
    Np_{j, k} + e^{-Np_{j, k}}
  } \, 1_{A_{j, k}},\\
  \wh R_1 &= \sum_{j, k} \int_{Np_{j, k} - e^{-Np_{j, k}}}^{Np_{j, k}} {
    \frac 1 {t + e^{-Np_{j, k}}}
  } \, dt \, 1_{A_{j, k}}, \text{ and}\\
  \wh R_2 &= \sum_{j, k} \int_{Np_{j, k}}^{S^{y_{j, k}}f + \ze_{j, k}} {
    \frac {(S^{y_{j, k}}f + \ze_{j, k} - t)^a} {(t + e^{_-Np_{j, k}})^{a + 1}}
  } \, dt \, 1_{A_{j, k}}.
\end{align*}
Note that for any $r \ge 1$, we have $E|\ze|^r \le 2e^{-Np}$. Thus, in all of
our earlier results, except for Corollaries \ref{C:CLT3} and \ref{C:CLT5}, if we
replace $S - Np$ with $S + \ze - Np$, then the proofs go through with minor
modifications. This includes our Berry-Essen inequality, \eqref{BE}. In fact, we
obtain one improvement, which is that $\wh R_1$ is exponentially small, so that
for any $r > 0$, there is a constant $C$ such that $\|\wh R_1\|_2 \le C/N^r$.

Hence, the law of large numbers, Theorem \ref{T:LLN}, holds with $Y_{n, m, N}f$
replaced by $\wh Y_{n, m, N}f$, and we have the following central limit theorem.

\begin{thm}\label{T:maxCLT}
  Let $a \in \bN$ and let $W$ be a white noise on $Z$. Define
  \begin{equation}\label{maxCLT}
    \wh Z_{n, m, N}f = \sqrt{nmN} \, \bigg(
      \wh Y_{n, m, N}f - X_{n, m}f - \sum_{r = 2}^a \frac {
        (-1)^r \mu_r(N e^{-X_{n, m}f})
      } {
        r (N e^{-X_{n, m}f} + e^{-Np_{j, k}})^r
      }
    \bigg).
  \end{equation}
  Then $\wh Z_{n, m, N}f \To \pi e^{Xf/2} \, W$ in $\cD'(Z)$ whenever $n, m, N
  \to \infty$ in such a way that $nm/N^a \to 0$.
\end{thm}

\begin{proof}
  If, in \eqref{maxCLT}, we replace $\mu$ with $\wh \mu$, then all of our
  earlier proofs go through with only minor modifications. Note that the sum
  involving $b$ is not needed because $\wh R_1$ is already exponentially small.
  Since $E|1_{\{S = 0\}} - e^{-\la}|^r$ is exponentially small, $\wh \mu(\la) -
  \mu(\la)$ is exponentially small. Hence, the result is also true as stated,
  with $\mu$ instead of $\wh \mu$.
\end{proof}

Aside from the above functional central limit theorem, the Berry-Esseen
inequality, \eqref{BE}, still holds with $Z_{n, m, N}f$ replaced by $\wh Z_{n,
m, N}f$. From Theorem \ref{T:maxCLT}, we obtain the following consequences, with
proofs analogous to Corollaries \ref{C:CLT3} and \ref{C:CLT5}.

\begin{cor}\label{C:maxCLT}
  Let $W$ be a white noise on $Z$. Then
  \begin{equation*}
    \sqrt{nmN} \left(
      \wh Y_{n, m, N}f - X_{n, m}f - \frac 1 {2N} e^{X_{n, m}f}
    \right) \To \pi e^{Xf/2} W
  \end{equation*}
  in $\cD'(Z)$ whenever $n, m, N \to \infty$ in such a way that $nm/N^3 \to 0$.
  Also,
  \begin{equation*}
    \sqrt{nmN} \left(
      \wh Y_{n, m, N}f - X_{n, m}f - \frac 1 {2N} e^{X_{n, m}f}
      + \frac 7 {12N^2} e^{2 X_{n, m}f}
    \right) \To \pi e^{Xf/2} W
  \end{equation*}
  in $\cD'(Z)$ whenever $n, m, N \to \infty$ in such a way that $nm/N^5 \to 0$.
\end{cor}

Now let $\wt S^{y_{j, k}}f$ have the same distribution as $S^{y_{j, k}}f$
conditioned to be positive, and define
\begin{equation*}
  \wt Y_{n, m, N}f = - \sum_{j, k} {
    \log \left( \frac{\wt S^{y_{j, k}}f} N \right) 1_{A_{j, k}}
  }.
\end{equation*}
We can construct $S^{y_{j, k}}f$ and $\wt S^{y_{j, k}}f$ in a coupled way so
that they are equal on the event $\{S^{y_{j, k}}f > 0\}$. For example, if $S_1,
S_2, \ldots$ are i.i.d.~$\Pois(Np_{j, k})$ and $\tau = \min\{i: S_i > 0\}$, then
$S_1 =_d S^{y_{j, k}}f$, $S_\tau =_d \wt S^{y_{j, k}}f$, and $S_1 = S_\tau$ on
$\{S_1 > 0\}$. Hence,
\begin{equation*}
  \wt Y_{n, m, N}f - \wh Y_{n, m, N}f = - \sum_{j, k} \left(
    \log \left( \frac{\wt S^{y_{j, k}}f} N \right)
    - \log \left( \frac 1 N \right)
  \right) 1_{\{S^{y_{j, k}}f = 0\}} 1_{A_{j, k}}.
\end{equation*}
Since $P(S^{y_{j, k}}f = 0)$ is exponentially small, $E\|\wt Y_{n, m, N}f - \wh
Y_{n, m, N}f\|_2^2$ is exponentially small. Hence, for any $r > 0$, there is a
constant $C$ such that $\|\wt Y_{n, m, N}f - \wh Y_{n, m, N}f\|_{L^2(S \times
\Om)} \le C/N^r$. It follows that if we define
\begin{equation*}
  \wt Z_{n, m, N}f = \sqrt{nmN} \, \bigg(
    \wt Y_{n, m, N}f - X_{n, m}f - \sum_{r = 2}^a \frac {
      (-1)^r \mu_r(N e^{-X_{n, m}f})
    } {
      r (N e^{-X_{n, m}f} + e^{-Np_{j, k}})^r
    }
  \bigg),
\end{equation*}
then both Theorem \ref{T:maxCLT} and Corollary \ref{C:maxCLT} still hold with
$\wh Y_{n, m, N}f$ replaced by $\wt Y_{n, m, N}f$ and $\wh Z_{n, m, N}f$
replaced by $\wt Z_{n, m, N}f$. We also have the law of large numbers, Theorem
\ref{T:LLN}, and the Berry-Esseen inequality, \eqref{BE}, with $Y_{n, m, N}f$
and $Z_{n, m, N}f$ replaced by $\wt Y_{n, m, N}f$ and $\wt Z_{n, m, N}f$,
respectively.

\bibliographystyle{plain}
\bibliography{tomography}
\addcontentsline{toc}{section}{References}

\end{document}